\documentclass[a4paper,twopage,reqno,12pt]{amsart}
\usepackage[top=30mm,right=30mm,bottom=30mm,left=30mm]{geometry}

\usepackage{amsmath}
\usepackage{amssymb}
\usepackage{amsfonts}
\usepackage{amsthm}
\usepackage{amstext}
\usepackage{amsbsy}
\usepackage{amsopn}
\usepackage{amscd}
\usepackage{pdfsync}
\usepackage{graphicx}
\usepackage{enumerate}
\usepackage{color}
\usepackage[hyperpageref]{backref}
\usepackage[colorlinks]{hyperref}
\hypersetup{
colorlinks=true,%
citecolor=green,%
filecolor=black,%
linkcolor=blue,%
urlcolor=magenta
}
\usepackage{caption}
\usepackage{subcaption}
\usepackage{pgf,tikz}
 \usetikzlibrary{arrows,shapes}
  \tikzstyle{vertex}=[circle,fill=blue,inner sep=2pt]
  \tikzstyle{edge style}=[line width=.5pt]
  \tikzstyle{label style}=[circle,fill=white,inner sep=1pt,font=\tiny]
  \tikzstyle{squirestyle}=[scale=.8,rounded corners,minimum height=5cm,minimum width=7cm,draw]
  \tikzstyle{circlestyle}=[scale=.8,circle,minimum height=3cm, draw]
  \tikzstyle{ellipsestyle}=[scale=.8,ellipse,dash pattern=on 5pt off 5pt,minimum height=2.3cm,minimum width=4.5cm, draw]

 \newtheorem{theorem}{Theorem}[section]
 
 \newtheorem{lemma}[theorem]{Lemma}
 \newtheorem{proposition}[theorem]{Proposition}
 
 \theoremstyle{definition}
 
 \theoremstyle{remark}
 
 \numberwithin{equation}{section}

 \newcommand{\D}{\textsf{D}}

\newcommand{\Syl}{\textsf{Syl}}
\newcommand{\Aut}{\textsf{Aut}}
\newcommand{\Out}{\textsf{Out}}

\newcommand{\N}{\mathbf{N}}

\renewcommand{\varphi}{\phi}
\renewcommand{\leq}{\leqslant}
\renewcommand{\geq}{\geqslant}

\begin{document}

\title[Characterization of some $L_{3}(q)$]{A characterization of some projective special linear groups}

\author[A. Daneshkhah]{Ashraf Daneshkhah$^{\ast}$}
 \thanks{$^{\ast}$Corresponding author: Ashraf Daneshkhah}
 \address{%
 A. Daneshkhah, Department of Mathematics,
 Faculty of Science,
 Bu-Ali Sina University, Hamedan, Iran}
 \email{adanesh@basu.ac.ir and daneshkhah.ashraf@gmail.com (Gmail preferred)}

\author[Y. Jalilian]{Younes Jalilian}
 \address{
 Y. Jalilian, Department of Mathematics,
 Faculty of Science,
 Bu-Ali Sina University, Hamedan, Iran }
 \email{ y.jalilian91@basu.ac.ir }




\begin{abstract}
 In this paper, we show that projective special linear groups $S:=L_3(q)$ with $q$ less than $100$ are uniquely determined by their orders and degree patterns of their prime graphs. Indeed, we prove that if $G$ is a finite group whose order and degree pattern of its prime graph is the same as the order and the degree pattern of $S$, then $G$ is isomorphic to $S$.
\end{abstract}
\date{February 9, 2016}
\maketitle
\section{Introduction}\label{s:intro}
Let $G$ be a finite group and $\pi(G)=\{p_1, p_2,\ldots, p_k\}$ be the set of primes dividing the order of $G$. The \emph{prime graph} $\Gamma(G)$ of  $G$ is a simple graph whose vertex set is $\pi(G)$ and two distinct primes $p$ and $q$ in $\pi(G)$ are adjacent if and only if there exists an element of order $pq$ in $G$. If $p$ is adjacent  to $q$, we write $p\sim q$. For $p\in\pi(G)$, the \emph{degree} of $p$ is the number $\deg_{G}(p):=| \{q\in\pi(G)\mid p\sim q\} |$. The \emph{degree pattern} $\D(G)$ of $G$ is the $k$-tuple $(\deg_{G}(p_1), \deg_{G}(p_2),\ldots,\deg_{G}(p_k))$ in which $p_1<p_2<\cdots <p_k$. A group $G$ is said to be \emph{OD-characterizable} if there exists exactly one isomorphic class of finite groups with the same order and degree pattern as $G$.

Darafsheh et al in \cite{Darafsheh} studied the quantitative structure of finite groups using their degree patterns and proved that if $|\pi((q^2+q+1)/d)|=1$, where $d=(3,q-1)$ and $q\geq 5$, then $L_3(q)$ is OD-characterizable. In \cite{Rezaeezadeh1}, it is shown that $L_3(25)$ is OD-characterizable. Also the authors in \cite{Rezaeezadeh} proved that the simple group $L_3(2^n)$ with $n\in\{4,5,6,7,8,10,12\}$ is OD-characterizable. Note in passing that all finite simple groups whose orders are less than $10^8$ are OD-characterizable, see \cite{Zhang1}. In this paper, we show that $L_{3}(q)$, where $q$ is a prime power less than $100$, is uniquely determined by its order and the degree pattern of its prime graph, that is to say,

\begin{theorem}\label{thm:main}
 Let $G$ be a finite group, and let $q$ be a prime power less than $100$. If $|G|=|L_3(q)|$ and $\D(G)=\D(L_3(q))$, then $G\cong L_3(q)$.
\end{theorem}

In order to prove Theorem~\ref{thm:main}, by \cite{Darafsheh,Rezaeezadeh,Rezaeezadeh1}, we only need to show that $L_{3}(q)$ is OD-characterisable for $q\in\{11,23,29,37,47,49, 53,61,67,79,81,83\}$.

Throughout this article all groups are finite. The \emph{spectrum} $\omega(G)$ of a group $G$ is the set of orders of its elements, and  $\mu(G)$ is the set of elements of $\omega(G)$ that are maximal with respect to divisibility relation. Let $t(G)$ be the number of \emph{connected components} of $\Gamma(G)$, and set $\pi_{i}:=\pi_i(G)$, for $i=1,\ldots,t(G)$. When $|G|$ is even, we always assume that $2\in\pi_1$. For a positive integer $n$, the set of all primes dividing $n$ is denoted $\pi(n)$, and recall that  $\pi(G):=\pi(|G|)$. The \emph{$p$-part} of $n$ is denoted by $n_{p}$, that is to say, $n_{p}=p^{\alpha}$ if $p^{\alpha}\mid n$ but $p^{\alpha+1}\nmid n$. For $i=1,\ldots, t(G)$, a positive integer $m_{i}$ with $\pi(m_i)=\pi_i$ is called an \emph{order component} of $G$. Then $|G|$ can be expressed as a product of $m_1$, $m_2$, \ldots, $m_{t(G)}$. If $m_i$ is odd, then $m_{i}$ is called an \emph{odd order component} of $G$. All further definitions and notation are standard and can be found in \cite{ATLAS,Gorenstein}.

\section{Preliminaries}\label{sec:pre}

In this section, we mention some useful result to be used in proof of Theorem~\ref{thm:main}.

\begin{lemma}\cite[Theorem 1]{Higman}\label{lem:solv}
Let $G$ be a finite solvable group all of whose elements are of prime power order. Then $|\pi (G)|\leq 2$.
\end{lemma}

\begin{lemma}\cite[Theorem 10.3.1]{Gorenstein}\label{lem:Thombson}
Let $G$ be a Frobenius group with kernel $K$ and complement $H$ . Then
\begin{enumerate}
  \item[(a)] $K$ is a nilpotent group.
  \item[(b)] $|H|$ divides $|K|-1$.
  \item[(c)] Every subgroup of $H$ of order $pq$ with $p$ and $q$ primes (not necessarily distinct), is cyclic.
  \item[(d)] Every Sylow subgroup of $H$ of odd order is cyclic and a Sylow $2$-subgroup of $H$ is either cyclic, or a generalized quaternion group.
  \item[(e)] If $H$ is non-solvable group, then $H$ has a subgroup of index at most $2$ isomorphic to $SL(2,5)\times M$, where $M$ has cyclic Sylow $p$-subgroups and $\left(| M| ,30\right)=1$.
\end{enumerate}
\end{lemma}

A group $G$ is said to be a \emph{$2$-Frobenius} group if there exists a normal series $1\unlhd H\lhd K\unlhd G$ such that $K$ and $G/H$ are Frobenius groups with kernel $H$ and $K/H$, respectively.

\begin{lemma}\cite[Theorem 2]{Higman}\label{lem:twofrobenius}
Let $G$ be 2-Frobenius group of even order. Then
\begin{enumerate}
\item[(a)] $t(G)=2$, $\pi_1(G)=\pi(H)\cup\pi(G/K)$ and $\pi_2(G)=\pi(K/H)\}$.
\item[(b)] $G/K$ and $K/H$ are cyclic, $|G/K|\mid|\Aut(K/H)|$, and $( |G/K|,|K/H|)=1$.
\item[(c)] $H$ is nilpotent and $G$ is solvable group.
\end{enumerate}
\end{lemma}

\begin{lemma}\cite{Williams} \label{lem:william}
Let $G$ be a finite group with $t(G)\geq 2$ and $2\in \pi_{1}$. Then $G$ is one of the following groups:
\begin{enumerate}
\item[(a)] $G$ is a Frobenius or 2-Frobenuis group.
\item[(b)] $G$ has a normal series $1\unlhd H\lhd K\unlhd G$, such that $H$ is nilpotent $\pi_{1}$-group, $G/K$
 is a $\pi_{1}$-group, and $K/H$ is non-Abelian finite  simple group such that $\left| G/H\right| \mid \left| \Aut(K/H)\right|$. Moreover, any odd order component of $G$ is also an odd order component of $K/H$.
\end{enumerate}
\end{lemma}

An \emph{independent set} of a graph $\Gamma$ is a set of vertices of $\Gamma$  no two of which are adjacent. The \emph{independence number} $\alpha(\Gamma)$ of $\Gamma$ is the maximum cardinality of an independent set among all independent sets of $\Gamma$. For convenience, if $G$ is a group, we set $\alpha(G):=\alpha(\Gamma(G))$. Moreover, for a vertex $r\in\pi(G)$, let $\alpha(r,G)$ denote the maximal number of vertices in independent sets of $\Gamma(G)$ containing $r$.

\begin{lemma}\cite[Theorem 1]{Vasilev}\label{almost simple}
Let $G$ be a finite group with $\alpha(G)\geq 3$ and $\alpha(2,G)\geq 2$, and let $K$ be the maximal normal solvable subgroup of $G$. Then the quotient group $G/K$ is an almost simple group, that is, there exists a non-Abelian finite  simple group $S$ such that $S\leq G/K\leq \Aut(S)$.
\end{lemma}
\begin{lemma}\label{lem:non-sol}
Let $G$ be a finite group of even order with $\alpha(G)\geq 3$. Then $G$ is non-solvable, and so it is not a $2$-Frobenius group. If, moreover, $|G|_3\neq 3$ or $|G|_5\neq 5$, then $G$ is not a Frobenius group.
\end{lemma}
\begin{proof}
Suppose that $\alpha(G)\geq 3$. If $G$ were solvable, then it would have a Hall $\{p,q,r\}$-subgroup $T$ with $\{p,q,r\}$ an independent subset of $\Gamma(G)$. Then $p$, $q$ and $r$ are not pairwise adjacent in $\Gamma(G)$, and so each element of $T$ is of prime power order. Since $T$ is solvable, it follows from Lemma \ref{lem:solv} that $|\pi (T)|\leq 2$, which is a contradiction. Therefore, $G$ is non-solvable, and so by Lemma \ref{lem:twofrobenius}, $G$ is not a $2$-Frobenius. Let now $G$ be a Frobenius group with complement $H$ and kernel $K$. Since $G$ is non-solvable, it follows from  Lemma \ref{lem:Thombson} that $H$ has a normal subgroup $H_{0}$ with $| H:H_{0}|\leq 2$ such that $H_{0}=SL(2,5)\times Z$, where $(| Z|,30)=1$. Then $|H|=2^a\cdot 3\cdot 5\cdot |Z|$ with $a=3,4$. Therefore, $|G|_{3}=3$ and $|G|_{5}=5$.
\end{proof}

\begin{lemma}\label{lem:alph3}
Let $\Gamma(G)$ be the prime graph of a group $G$ with two vertices of degree $1$. Then $\alpha(G)\geq 3$ if one of the following holds:
\begin{enumerate}
  \item[(a)] $|\pi(G)|= 6$ and $\Gamma(G)$ has at least two vertices of degree $2$;
  \item[(b)] $|\pi(G)|\geq 7$ and $\Gamma(G)$ has at least two vertices of degree $3$.
\end{enumerate}
\end{lemma}
\begin{proof}
Suppose that $p_{1}$  and $p_{2}$ are two vertices of $\Gamma(G)$ with $\deg(p_{i})=1$, for $i=1,2$. Assume that $p_{1}$ is adjacent to $p_{2}$. If  $|\pi(G)|=6$, there are four vertices which are not adjacent to $p_{1}$ and $p_{2}$. Since in this case there are at least two vertices of degree $2$, there exist two non-adjacent vertices $p_{3}$ and $p_{4}$ in $\Gamma(G)$. Therefore, $\{p_{1},p_{3},p_{4}\}$ is an independent set of $\Gamma(G)$ which implies that $\alpha(G)\geq 3$. Similarly, in the case where $|\pi(G)|\geq 7$, there are at least five vertices which are not adjacent to $p_{1}$ and $p_{2}$, and since we have at least two vertices of degree $3$, we can find two non-adjacent vertices $p_{3}$ and $p_{4}$ in $\Gamma(G)$, and hence  $\{p_{1},p_{3},p_{4}\}$ is an independent set of $\Gamma(G)$, consequently,  $\alpha(G)\geq 3$. Assume now $p_{1}$ and $p_{2}$ are non-adjacent. Since $|\pi(G)|\geq 6$ and both $p_{1}$ and $p_{2}$ are of degree $1$, there exists a vertex $p_{3}$ which is not adjacent to $p_{1}$ and $p_{2}$. Thus $\{p_{1},p_{2},p_{3}\}$ is an independent set of $\Gamma(G)$, and hence $\alpha(G)\geq 3$.
\end{proof}

\begin{lemma}\cite[Lemma 2.8]{Shitian} \label{out}
Let $S$ be a finite non-abelian simple group, and let $p$ be the largest prime divisor of $|S|$ with $|S|_{p}=p$. Then $p\nmid |\Out(S)|$.
\end{lemma}

\section{Proof of main Result}\label{sec:main}

In this section, we prove Theorem~\ref{thm:main}. For convenience, in Table~\ref{tbl:simple}, we list the order, spectrum and degree pattern of $S:=L_3(q)$, where $q\in\{11,23,29,37,47,49$ $53,61,67,79,81,83\}$. In order to determine the degree pattern of $S$ as in the first column of Table~\ref{tbl:simple}, we use $\mu(S)$ (see \cite[Theorem 9]{Grechkoseeva}):
\begin{align*}
   \mu(S)=\left\{\dfrac{q^2+q+1}{(3,q-1)}, \dfrac{q^2-1}{(3,q-1)}, q-1, \dfrac{p(q-1)}{(3,q-1)}\right\}.
\end{align*}
Note that if $(3,q-1)=1$, then $\mu(S)=\{q^2+q+1, q^2-1,p(q-1)\}.$ We also note that the order of $S$ is
\begin{align*}
  |L_3(q)|=\dfrac{1}{(3,q-1)}q^3( q^2-1)(q^3-1).
\end{align*}

In what follows, we assume that $G$ is a finite group with $|G| =|S|$ and $\D(G)=\D(S)$, see Table~\ref{tbl:simple} below.

\begin{table}[h]
 \centering
 \small
  \caption{The order, spectrum and degree pattern of $L_3(q)$, for $q\in \{11,23,29,37,47,49,53,61,64,67,79,81,83\}$.}\label{tbl:simple}
  \begin{tabular}{llp{5cm}l}
   \hline
    \multicolumn{1}{c}{$S$} &  \multicolumn{1}{l}{$|S|$} &  \multicolumn{1}{l}{$\mu(S)$}&  \multicolumn{1}{l}{$\D(S)$} \\ \hline
    $L_3(11)$  &
    $2^4\cdot 3\cdot 5^2\cdot 7\cdot 11^3\cdot 19$ &
    $\{7\cdot 19,2^3\cdot 3\cdot 5,2\cdot 5\cdot 11\}$ &
    $(3,2,3,1,2,1)$\\
     $L_3(23)$  &
    $2^5\cdot 3\cdot 7\cdot 11^2\cdot 23^3\cdot 79$ &
    $\{7\cdot 79,2^4\cdot 3\cdot 11,2\cdot 11\cdot 23\}$ &
    $(3,2,1,3,2,1)$\\
     $L_3(29)$  &
    $2^5\cdot 3\cdot 5\cdot 7^2\cdot 13\cdot 29^3\cdot 67$ &
    $\{13\cdot 67,2^3\cdot 3\cdot 5\cdot 7,2^2\cdot 7\cdot 29\}$ &
    $(4,3,3,4,1,2,1)$\\
    $L_3(37)$  &
    $2^5\cdot 3^4\cdot 7\cdot 19\cdot 37^3\cdot 67$ &
    $\{7\cdot 67,2^3\cdot 3\cdot 19, 2^2\cdot 3^2,2^2\cdot 3\cdot 37\}$ &
    $(3,3,1,2,2,1)$\\
    $L_3(47)$&
    $2^6\cdot 3\cdot 23^2\cdot 37\cdot 47^3\cdot 61$&
    $\{37\cdot 61,2^5\cdot 3\cdot 23,2\cdot 23\cdot 47\}$&
    $(3,2,3,1,2,1)$\\
  $L_3(49)$ &
  $2^9\cdot 3^2\cdot 5^2\cdot 7^6\cdot 19\cdot 43$&
  $\{19\cdot 43,2^5\cdot 5^2, 2^4\cdot 3,2^4\cdot 7\}$&
  $(3,1,1,1,1,1)$\\
  $L_3(53)$ &
  $2^5\cdot 3^3\cdot 7\cdot 13^2\cdot 53^3\cdot 409$&
  $\{7\cdot 409,2^3\cdot 3^3\cdot 13,2^2\cdot 13\cdot 53\}$&
  $(3,2,1,3,2,1)$\\
 $L_3(61)$ &
 $2^5\cdot 3^2\cdot 5^2\cdot 13\cdot 31\cdot 61^3\cdot 97$&
 $\{13\cdot 97,2^3\cdot 5\cdot 31, 2^2\cdot 3\cdot 5,2^2\cdot 5\cdot 61\}$&
 $(4,2,4,1,2,2,1)$\\
 $L_3(67)$ &
 $2^4\cdot 3^2\cdot 7^2\cdot 11^2\cdot 17\cdot 31\cdot 67^3$&
 $\{7^2\cdot 31,2^3\cdot 11\cdot 17, 2\cdot 3\cdot 11,2\cdot 11\cdot 67\}$&
 $(4,2,1,4,2,1,2)$\\
 $L_3(79)$ &
 $2^6\cdot 3^2\cdot 5\cdot 7^2\cdot 13^2\cdot 43\cdot 79^3$&
 $\{7^2\cdot 43,2^5\cdot 5\cdot 13, 2\cdot 3\cdot 13,2\cdot 13\cdot 79\}$&
 $(4,2,2,1,4,1,2)$\\
 $L_3(81)$ &
 $2^9\cdot 3^{12}\cdot 5^2\cdot 7\cdot 13\cdot 41\cdot 73$&
 $\{7\cdot 13\cdot 73,2^5\cdot 5\cdot 41,2^4\cdot 3\cdot 5\}$&
 $(3,2,3,2,2,2,2)$\\
 $L_3(83)$ &
 $2^4\cdot 3\cdot 7\cdot 19\cdot 41^2\cdot 83^3\cdot 367$&
 $\{19\cdot 367,2^3\cdot 3\cdot 7\cdot 41,2\cdot 41\cdot 83\}$&
 $(4,3,3,1,4,2,1)$\\ \hline
\end{tabular}
\end{table}
\begin{proposition}\label{prop:1}
If $|G| =|L_{3}(11)|$ and $\D(G)=\D(L_3(11))$, then $G\cong L_3(11)$.
\end{proposition}
\begin{proof}
By Table~\ref{tbl:simple}, we have that  $|G|=2^4\cdot 3\cdot 5^2\cdot 7\cdot 11^3\cdot 19$ and $\D(G)=(3,2,3,1,2,1)$. Then Lemma~\ref{lem:alph3} implies that $\alpha(G)\geq3$. Furthermore, $\alpha(2,G)\geq2$ as $\deg(2)=3$ and $|\pi(G)|=6$. By Lemma \ref{almost simple}, there is a non-Abelian finite  simple group $S$ such that $S\leq G/K\leq \Aut(S)$, where $K$ is a maximal normal solvable subgroup of $G$.

We show that $\pi(K)\subseteq \{2,3,5,11\}$. Assume the contrary. Then $19\in \pi(K)$. We show that $p$ is adjacent to $19$, where $(p,a)\in\{(5,1),(5,2),(7,1)\}$. If $p\in\pi(K)$, then $K$ contains an abelain Hall subgroup of order $p^{a}\cdot 19$, and so $p$ is adjacent to $19$. If $p\not \in \pi(K)$, then by Frattini argument $G=K\N_G(P)$, where $P$ is a Sylow $19$-subgroup of $K$. Thus $\N_G(P)$ contains an element of order $p$, say $x$. So $P\langle  x \rangle$ is a cyclic subgroup of $G$ of order $p\cdot 19$ concluding that $p$ is adjacent to $19$. Therefore, both $5$ and $7$ are adjacent to $19$, and hence $19$ is of degree at least $2$, which is a contradiction. Similarly, we can show that $7\notin\pi(K)$. Hence $\pi(K)\subseteq \{2,3,5,11\}$.

We now prove that $S$ is isomorphic to $L_3(11)$. Note by Lemma \ref{out} that $19\notin\pi(\Out(S))$. Then $19\notin\pi(K)\cup\pi(\Out(S))$, and so $19\in\pi(S)$. Now by \cite[Table 1]{ZAVARNITSIN}, $S$ is isomorphic to one of the simple groups $J_1$ and $L_3(11)$.

If $S$ were isomorphic to $J_1$, then $|S|$ would be $2^3\cdot 3\cdot 5\cdot 7\cdot 11\cdot 19$, and since $\Out(S)=1$, we must have $|K|=2\cdot 5\cdot 11^2$. Let $P\in \Syl_{11}(K)$, and let $r\in \{7,19\}$. By Frattini argument, $G=K\N_G(P)$, and so  $\N_G(P)$ contains an element of order $r$, say $x$. Since $P$ is normal in $K$ and $P\cap \langle  x \rangle=1$, $L:=P\langle x \rangle$ is a subgroup of $K$ of order $r\cdot 11^{2}$. Since also $L$ is Abelian, it has an element of order $r\cdot 11$. This shows that both $7$ and $19$ are adjacent to $11$. Note that the degree of $11$ is two, and $7$ and $19$ are of degree one. Thus $2$ can not be adjacent to none of $7$, $11$ and $19$. Since $|\pi(G)|=6$, the degree of $2$ is at most $2$, which is a contradiction.

Therefore, $S$ is isomorphic to $L_3(11)$, and hence $L_3(11)\leq G/K\leq \Aut(L_3(11))$. Note that $|G|=|L_3(11)|$. Thus $K=1$, and hence $G$ is isomorphic to $L_3(11)$.
\end{proof}

\begin{proposition}\label{prop:2}
If $|G| =|L_{3}(23)|$ and $\D(G)=\D(L_3(23))$, then $G\cong L_3(23)$.
\end{proposition}
\begin{proof}
According to Table~\ref{tbl:simple}, we have that $|G|=2^5\cdot 3\cdot 7\cdot 11^2\cdot 23^3\cdot 79$ and $\D(G)=(3,2,1,3,2,1)$. Then by Lemma~\ref{lem:alph3}, we conclude that $\alpha(G)\geq3$. Since $\deg(2)=3$ and $|\pi(G)|=6$, $\alpha(2,G)\geq2$. Therefore, by Lemma \ref{almost simple}, there is a non-Abelian finite  simple group $S$ such that $S\leq G/K\leq \Aut(S)$, where $K$ is a maximal normal solvable subgroup of $G$.

We claim that $79 \notin\pi(K)$. Assume the contrary. We show that $p$ is adjacent to $79$, where $(p,a)\in\{(7,1),(11,1),(11,2)\}$. If $p\in\pi(K)$, then $K$ contains an abelain Hall subgroup of order $p^{a}\cdot 79$ which implies that $p$ is adjacent to $79$. If $p\not \in \pi(K)$, then by Frattini argument $G=K\N_G(P)$, where $P$ is a Sylow $79$-subgroup of $K$, and so $\N_G(P)$ contains an element $x$ of order $p$. Note that $P\langle x \rangle$ is a cyclic subgroup of $G$ of order $p\cdot 79$. Then $p$ is adjacent to $79$. Hence, both $7$ and $11$ are adjacent to $79$, and consequently, degree of $79$ is at least $2$, which is a contradiction. Similarly, we can show that $7\notin\pi(K)$. Therefore $\pi(K)\subseteq \{2,3,7,11,23\}$, by Lemma \ref{out}, we have that $79\notin\pi(\Out(S))$. Then $79\notin\pi(K)\cup\pi(\Out(S))$, and so $79\in\pi(S)$. Therefore by \cite[Table 1]{ZAVARNITSIN}, $S$ is isomorphic to $L_3(23)$. Since $|G|=|L_3(23)|$, we must have $K=1$,  and hence $G$ is isomorphic to $L_3(23)$.
\end{proof}

\begin{proposition}\label{prop:3}
If $|G| =|L_{3}(29)|$ and $\D(G)=\D(L_3(29))$, then $G\cong L_3(29)$.
\end{proposition}
\begin{proof}
It follows from Table~\ref{tbl:simple} that $|G|=2^5\cdot 3\cdot 5\cdot 7^2\cdot 13\cdot 29^3\cdot 67$ and $\D(G)=(4,3,3,4,1,2,1)$. Then Lemma \ref{lem:alph3} implies that  $\alpha(G)\geq3$. Furthermore, $\alpha(2,G)\geq2$ as $\deg(2)=4$ and $|\pi(G)|=7$. Therefore, Lemma \ref{almost simple} implies that there is a non-Abelian finite  simple group $S$ such that $S\leq G/K\leq \Aut(S)$, where $K$ is a maximal normal solvable subgroup of $G$.

We show that $67\not \in \pi(K)$. Assume the contrary. Then $K$ has element of order $67$. We show that $p$ is adjacent to $67$ for all $p\in\{5,13\}$. If $p\in\pi(K)$, then we consider a cyclic Hall subgroup of order $p\cdot 67$ of $K$, and so $p$ and $73$ are adjacent. If $p\not \in \pi(K)$, then we apply Frattini argument and have that $G=K\N_G(P)$, where $P$ is a Sylow $67$-subgroup of $K$. Thus $\N_G(P)$ contains an element $x$ of order $p$. Now $P\langle x \rangle$ is a cyclic subgroup of $G$ of order $p\cdot 67$ which again implies that $p$  and $67$ are adjacent. Therefore, both $5$ and $13$ are adjacent to $67$, and hence the degree of $67$ must be at least $2$, which is a contradiction. Then $\pi(K)\subseteq \{2,3,5,7,13,29\}$. By Lemma \ref{out}, $67\notin\pi(\Out(S))$, then $67\notin\pi(K)\cup\pi(\Out(S))$, and so $67\in\pi(S)$. Using \cite[Table 1]{ZAVARNITSIN} we observe that $S$ is isomorphic to $L_3(29)$. Since $|G|=|L_3(29)|$, we conclude that $K=1$, and hence $G$ is isomorphic to $L_3(29)$.
\end{proof}

\begin{proposition}\label{prop:4}
If $|G| =|L_{3}(37)|$ and $\D(G)=\D(L_3(37))$, then $G\cong L_3(37)$.
\end{proposition}
\begin{proof}
Note by Table~\ref{tbl:simple} that  $|G|=2^5\cdot 3^4\cdot 7\cdot 19\cdot 37^3\cdot 67$ and $\D(G)=(3,3,1,2,2,1)$. Then by Lemma~\ref{lem:alph3}, we must have $\alpha(G)\geq3$. Furthermore, $\alpha(2,G)\geq2$ since $\deg(2)=3$ and $|\pi(G)|=6$. By Lemma \ref{almost simple}, there is a non-Abelian finite  simple group $S$ such that $S\leq G/K\leq \Aut(S)$, where $K$ is a maximal normal solvable subgroup of $G$. We show that $67\not \in \pi(K)$. Assume the contrary. $67 \in \pi(K)$. We prove that $p$ would be adjacent to $67$, where $p\in\{7,19\}$. If $p\in\pi(K)$, then $K$ contains a cyclic Hall subgroup of order $p\cdot 67$, and so $p$ is adjacent to $67$. If $p\not \in \pi(K)$, then it follows from   Frattini argument that $G=K\N_G(P)$, where $P$ is a Sylow $67$-subgroup of $K$, and so $\N_G(P)$ has an element $x$ of order $p$. Thus $P\langle x \rangle$ is a cyclic subgroup of $G$ of order $p\cdot 67$. Therefore both $7$ and $19$ are adjacent to $67$ which contradicts the fact that the degree of $67$ is $1$. Therefore, $67\notin\pi(K)$, and hence $\pi(K)\subseteq \{2,3,7,19,37\}$. Now we prove that $S\cong L_3(37)$. By Lemma \ref{out}, $67\notin\pi(\Out(S))$, then $67\notin\pi(K)\cup\pi(\Out(S))$, and so $67\in\pi(S)$. Therefore by \cite[Table 1]{ZAVARNITSIN}, $S$ is isomorphic to $L_3(37)$ as claimed. Since now $|G|=|L_3(37)|$, we must have $K=1$, and hence $G$ is isomorphic to $L_3(37)$.
\end{proof}

\begin{proposition}\label{prop:5}
If $|G| =|L_{3}(47)|$ and $\D(G)=\D(L_3(47))$, then $G\cong L_3(47)$.
\end{proposition}
\begin{proof}
By Table~\ref{tbl:simple}, we have that $|G|=2^6\cdot 3\cdot 23^2\cdot 37\cdot 47^3\cdot 61$ and $\D(G)=(3,2,3,1,2,1)$. It follows from Lemma \ref{lem:alph3} that $\alpha(G)\geq3$. Note that $\deg(2)=3$ and $|\pi(G)|=6$. Then $\alpha(2,G)\geq2$, and so by Lemma \ref{almost simple}, there is a non-Abelian finite  simple group $S$ such that $S\leq G/K\leq \Aut(S)$, where $K$ is a maximal normal solvable subgroup of $G$.
We claim that $61 \not \in \pi(K)$. Assume the contrary. Then $19\in \pi(K)$. We show that $p$ is adjacent to $61$, where $(p,a)\in\{(23,1),(23,2),(37,1)\}$. If $p\in\pi(K)$, then $K$ contains an abelain Hall subgroup of order $p^{a}\cdot 61$, and so $p$ is adjacent to $61$. If $p\not \in \pi(K)$, then by Frattini argument $G=K\N_G(P)$, where $P$ is a Sylow $61$-subgroup of $K$, and so $\N_G(P)$ contains an element $x$ of order $p$. Note that $P\langle  x \rangle$ is a cyclic subgroup of $G$ of order $p\cdot 61$. Then $p$ is adjacent to $61$. Therefore, $61$ is adjacent to both $23$ and $37$ in $\Gamma(G)$, and hence the degree of $61$ is at least $2$, which is a contradiction. Hence $61\notin\pi(K)$. Therefore, $\pi(K)\subseteq \{2,3,23,37,47\}$, and hence $61\notin\pi(\Out(S))$ by Lemma \ref{out}. Then $61\notin\pi(K)\cup\pi(\Out(S))$, and so $61\in\pi(S)$. Now by \cite[Table 1]{ZAVARNITSIN}, $S$ is isomorphic to $L_3(47)$, and so $L_3(47)\leq G/K\leq \Aut(L_3(47))$. Note that $|G|=|L_3(47)|$. Then $K=1$, and hence $G$ is isomorphic to $L_3(47)$.
\end{proof}

\begin{proposition}\label{prop:6}
If $|G| =|L_{3}(49)|$ and $\D(G)=\D(L_3(49))$, then $G\cong L_3(49)$.
\end{proposition}
\begin{proof}
According to Table~\ref{tbl:simple}, $|G|=2^9\cdot 3^2\cdot 5^2\cdot 7^6\cdot 19\cdot 43$ and $\D(G)=(3,1,1,1,1,1)$. Then we observe that $\Gamma(G)$ is the graph as in Figure \ref{fig:49} in which $\{a, b, c, d, e\} = \{3, 5, 7, 19, 43\}$.
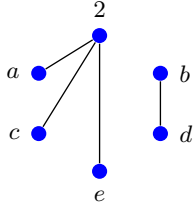
\begin{figure}[h]
  \centering
  \caption{Possibilities for the prime graph of $G$ in Proposition~\ref{prop:6}.}\label{fig:49}
 \begin{tikzpicture}[scale=1]
    \node[vertex] (1) [label=above:{\scriptsize $2$}] at (0,.9){};
    \node[vertex] (2) [label=left :{\scriptsize $a$}] at (-.8,0.4){};
    \node[vertex] (5) [label=right:{\scriptsize $b$}] at (.8,0.4){};
    \node[vertex] (3) [label=left :{\scriptsize $c$}] at (-.8,-.4){};
    \node[vertex] (6) [label=right:{\scriptsize $d$}] at (0.8,-0.4){};
    \node[vertex] (4) [label=below :{\scriptsize $e$}] at (0,-.9){};
    \path[edge style]
    (1) edge [] node [] {} (2)
        edge [] node [] {} (3)
        edge [] node [] {} (4)
    (5) edge [] node [] {} (6);
 \end{tikzpicture}
\end{figure}
We also observe that $t(G)=2$ and $\{a,b,e\}$ is an independent set. Thus $\alpha(G)\geq 3$. It is also easily seen that $\alpha(2, G)\geq 2$. Then by Lemma~\ref{lem:non-sol}, $G$ is neither Frobenius, nor $2$-Frobenius, and so Lemma~\ref{lem:william} implies that $G$ has a normal series $1\unlhd H\lhd K\unlhd G$ such that $K/H$ is a non-Abelian finite simple group. Since $|K/H|$ divides $|K|$, it divides $|G|$, and so by \cite[Table 1]{ZAVARNITSIN}, the factor group $K/H$ is isomorphic to one of the simple groups $S$ as in the first column of Table \ref{tbl:49} below.
\begin{table}[h]
 \centering
 \caption{Non-Abelian finite  simple groups $S$ whose order divides $|L_3(49)|$}\label{tbl:49}
 \small
\begin{tabular}{llcl}
 \multicolumn{1}{l}{$S$} & \multicolumn{1}{l}{$|S|$} & \multicolumn{1}{c}{$|\Out(S)|$} & \multicolumn{1}{l}{Primes in $\pi(H)$}\\ \hline
 $L_2(4)$ & $2^2\cdot 3\cdot 5$ & $2$& $7$, $19$, $43$ \\
 $L_2(9)$ & $2^3\cdot 3^2\cdot 5$ & $4$& $7$, $19$, $43$\\
 $L_2(7)$ & $2^3\cdot 3\cdot 7$ & $2$& $7$, $19$, $43$\\
 $L_2(8)$ & $2^3\cdot 3^2\cdot 7$ & $3$& $7$, $19$, $43$\\
 $A_7$ & $2^3\cdot 3^2\cdot 5\cdot 7$ & $2$& $7$, $19$, $43$\\
 $L_2(49)$ & $2^4\cdot 3\cdot 5^2\cdot 7^2$ & $4$& $7$, $19$, $43$\\
 $L_3(4)$ & $2^6\cdot 3^2\cdot 5\cdot 7$ & $12$& $7$, $19$, $43$\\
 $L_4(2)$ & $2^6\cdot 3^2\cdot 5\cdot 7$ & $2$& $7$, $19$, $43$\\
 $S_4(7)$ & $2^8\cdot 3^2\cdot 5^2\cdot 7^4$ & $2$& $7$, $19$, $43$\\
 $L_2(19)$ & $2^2\cdot 3^2\cdot 5\cdot 19$ & $2$& $2$, $7$, $43$\\
 $L_3(7)$ & $2^5\cdot 3^2\cdot 7^3\cdot 19$ & $6$& $2$, $7$, $43$\\
 $U_3(7)$ & $2^7\cdot 3\cdot 7^3\cdot 43$ & $2$& $2$, $5$, $7$\\
 $L_2(7^3)$ & $2^3\cdot 3^2\cdot 7^3\cdot 19\cdot 43$ & $6$& $2$, $5$, $7$\\
 $L_3(49)$ & $2^9\cdot 3^2\cdot 5^2\cdot 7^6\cdot 19\cdot 43$ & $12$& - \\
\end{tabular}
\end{table}

If $K/H$ is isomorphic to one of the groups $S$ listed in the first column of Table~\ref{tbl:49} except $L_{3}(49)$, then $\pi(H)$ consists of three primes as in the third column of the same table. Since $H$ is nilpotent, it follows that the prime graph of $G$ has a triangle, which is a contradiction. Therefore, $K/H\cong L_3(49)$. As $L_3(49)\leq G/H\leq \Aut(L_3(49))$ and $|G|=|L_3(49)|$, we conclude that $|H|=1$, and hence $G$  is isomorphic to $L_3(49)$.
\end{proof}

\begin{proposition}\label{prop:7}
If $|G| =|L_{3}(53)|$ and $\D(G)=\D(L_3(53))$, then $G\cong L_3(53)$.
\end{proposition}
\begin{proof}
By Table~\ref{tbl:simple}, $|G|=2^5\cdot 3^3\cdot 7\cdot 13^2\cdot 53^3\cdot 409$ and $\D(G)=(3,2,1,3,2,1)$. Now by applying Lemma~\ref{lem:alph3}, we must have  $\alpha(G)\geq3$. Furthermore, $\alpha(2,G)\geq2$ since $\deg(2)=3$ and $|\pi(G)|=6$. So by Lemma \ref{almost simple}, there is a non-Abelian finite  simple group $S$ such that $S\leq G/K\leq \Aut(S)$, where $K$ is a maximal normal solvable subgroup of $G$.  We claim that $\pi(K)$ does not contain $409$. Assume the contrary.  Then $409 \in\pi(K)$. We show that $p$ is adjacent to $409$, where $(p,a)\in\{(7,1),(13,1),(13,2)\}$. If $p\in\pi(K)$, then $K$ contains an abelain Hall subgroup of order $p^{a}\cdot 409$, so $p$ and $409$ are adjacent. If $p\not \in \pi(K)$, then we apply Frattini argument and have that $G=K\N_G(P)$, where $P$ is a Sylow $409$-subgroup of $K$, and so $\N_G(P)$ has an element $x$ of order $p$. Now $P\langle x \rangle$ is a cyclic subgroup of $G$ of order $p\cdot 409$ concluding that $p$ and $409$ are adjacent. Thus both $7$ and $13$ are adjacent to $409$ in $\Gamma(G)$, and hence the degree of $409$ is at least $2$, which is a contradiction. Therefore, $409\notin\pi(K)$, and hence it follows from Lemma \ref{out} that $409\notin\pi(\Out(S))$. Then $409\notin\pi(K)\cup\pi(\Out(S))$, and so $409\in\pi(S)$. Therefore by \cite[Table 1]{ZAVARNITSIN}, $S$ is isomorphic to $L_3(53)$ and $L_3(53)\leq G/K\leq \Aut(L_3(53))$. Moreover, since $|G|=|L_3(53)|$, it follows that $K=1$, and hence $G\cong L_3(53)$.
\end{proof}

\begin{proposition}\label{prop:8}
If $|G| =|L_{3}(61)|$ and $\D(G)=\D(L_3(61))$, then $G\cong L_3(61)$.
\end{proposition}
\begin{proof}
It follows from Table ~\ref{tbl:simple} that $|G|=2^5\cdot 3^2\cdot 5^2\cdot 13\cdot 31\cdot 61^3\cdot 97$ and $\D(G)=(4,2,4,1,2,2,1)$. Then by Lemma \ref{lem:alph3},  $\alpha(G)\geq3$. Moreover, $\alpha(2,G)\geq2$ as $\deg(2)=4$ and $|\pi(G)|=7$. By Lemma \ref{almost simple}, there is a non-Abelian finite  simple group $S$ such that $S\leq G/K\leq \Aut(S)$, where $K$ is a maximal normal solvable subgroup of $G$. We show that $97 \notin\pi(K)$. Assume the contrary. Then $97\in\pi(K)$.  Let $p\in\{13,31\}$. If $p\in\pi(K)$, then $K$ contains a cyclic Hall subgroup of order $p\cdot 97$, and so $p$ is adjacent to $97$. If $p\not \in \pi(K)$, then by Frattini argument $G=K\N_G(P)$, where $P$ is a Sylow $97$-subgroup of $K$, and so $\N_G(P)$ has an element $x$ of order $p$. Now $P\langle x \rangle$ is a cyclic subgroup of $G$ of order $p\cdot 97$ which implies that $p$ is adjacent to $97$. Therefore, $97$ is adjacent to $13$ and $31$ which is a contraiction as $97$ is of degree $1$. Thus $97 \notin\pi(K)$. Now by Lemma \ref{out}, $97\notin\pi(\Out(S))$, then $97\notin\pi(K)\cup\pi(\Out(S))$, and so $97\in\pi(S)$. Therefore by \cite[Table 1]{ZAVARNITSIN}, $S$ is isomorphic to $L_3(61)$ and $L_3(61)\leq G/K\leq \Aut(L_3(61))$. Moreover, since $|G|=|L_3(61)|$, we have that $K=1$, and hence  $G$ is isomorphic to $L_3(61)$.
\end{proof}

\begin{proposition}\label{prop:9}
If $|G| =|L_{3}(67)|$ and $\D(G)=\D(L_3(67))$, then $G\cong L_3(67)$.
\end{proposition}
\begin{proof}
By Table ~\ref{tbl:simple}, we have   $|G|=2^4\cdot 3^2\cdot 7^2\cdot 11^2\cdot 17\cdot 31\cdot 67^3$ and  $\D(G)=(4,2,1,4,2,1,2)$. It follows from Lemma~\ref{lem:alph3} that $\alpha(G)\geq3$. Note that $\deg(2)=4$ and $|\pi(G)|=7$. Then $\alpha(2,G)\geq2$. It follows from Lemma \ref{almost simple} that there is a non-Abelian finite  simple group $S$ such that $S\leq G/K\leq \Aut(S)$, where $K$ is a maximal normal solvable subgroup of $G$. Since $|S|$ divides $|G/K|$, so does $|G|$, and so by \cite[Table 1]{ZAVARNITSIN}, $S$ isomorphic to one of groups in the first column of Table \ref{tbl:67}.

If $S$ is isomorphic to one of the groups  $L_2(7)$, $L_2(8)$, $L_2(17)$ and $L_2(67)$, then $7, 11, 31\in\pi(K)$. As $K$ is solvable, we can consider a Hall $\{7,31\}$-subgroup $K_{1}:=P_7P_{31}$ and a Hall $\{11,31\}$-subgroup $K_{2}:=P_{11}P_{31}$. Then $|K_{1}|=7^2\cdot 31$ and $|K_{2}|=11^2\cdot 31$, and consequently $K_{i}$, for $i=1,2$, is Abelian which implies that $31$ is adjacent to both $7$ and $11$, and so  $\deg(37)\geq 2$, which is a contradiction. Thus $S\cong L_3(67)$. Since $L_3(67)\leq G/K\leq \Aut(L_3(67))$ and $|G|=|L_3(67)|$, we conclude that $|K|=1$, and hence $G$ isomorphic to $L_3(67)$.
\end{proof}

\begin{table}
 \centering
 \caption{Non-Abelian finite  simple groups $S$ whose order divides $|L_3(67)|$}\label{tbl:67}
 \small
\begin{tabular}{llc}
 \multicolumn{1}{l}{$S$} & \multicolumn{1}{l}{$|S|$} & \multicolumn{1}{c}{$|\Out(S)|$} \\ \hline
 $L_2(7)$ & $2^3\cdot 3\cdot 7$ & $2$\\
 $L_2(8)$ & $2^3\cdot 3^2\cdot 7$ & $3$\\
 $L_2(17)$ & $2^4\cdot 3^2\cdot 17$ & $2$\\
 $L_2(67)$ & $2^2\cdot 3\cdot 11\cdot 17\cdot 67$ & $2$\\
 $L_3(67)$ & $2^4\cdot 3^2\cdot 7^2\cdot 11^2\cdot 17\cdot 31\cdot 67^3$ & $6$\\
\end{tabular}
\end{table}

\begin{proposition}\label{prop:10}
If $|G| =|L_{3}(79)|$ and $\D(G)=\D(L_3(79))$, then $G\cong L_3(79)$.
\end{proposition}
\begin{proof}
By Table ~\ref{tbl:simple}, $|G|=2^6\cdot 3^2\cdot 5\cdot 7^2\cdot 13^2\cdot 43\cdot 79^3$ and $\D(G)=(4,2,2,1,4,1,2)$. Then by applying Lemma~\ref{lem:alph3}, we conclude that $\alpha(G)\geq3$. Furthermore, $\alpha(2,G)\geq2$ since $\deg(2)=4$ and $|\pi(G)|=7$. By Lemma \ref{almost simple}, there is a non-Abelian finite  simple group $S$ such that $S\leq G/K\leq \Aut(S)$, where $K$ is a maximal normal solvable subgroup of $G$. We show that $\pi(K)$ does not contain $43$. Assume the contrary. Then $43\in\pi(K)$. Let $(p,a)\in\{(5,1),(13,1),(13,2)\}$. If $p\in\pi(K)$, then $K$ contains an abelian Hall subgroup of order $p\cdot 43$, and so $p$ and $43$ are adjacent. If $p\not \in \pi(K)$, then by Frattini argument $G=K\N_G(P)$, where $P$ is a Sylow $43$-subgroup of $K$. This shows that $\N_G(P)$ contains an element $x$ of order $p$, and so $P\langle x \rangle$ is a cyclic subgroup of $G$ of order $p\cdot 43$ concluding that $p$ is adjacent to $43$. Therefore, $5$ and $13$ are adjacent to $43$, and so $43$ has degree at least $2$, which is a contradiction. Therefore, $43\notin \pi(K)$. We prove that $S\cong L_3(79)$. Since $|S|$ divides $G/K$, so does $|G|$, and so by \cite[Table 1]{ZAVARNITSIN}, $S$ isomorphic to one of groups in Table \ref{tbl:79} below.
\begin{table}
 \centering
 \caption{Non-Abelian finite  simple groups $S$ whose order divides $|L_3(79)|$.}\label{tbl:79}
 \small
\begin{tabular}{llc}
  \multicolumn{1}{l}{$S$} & \multicolumn{1}{l}{$|S|$} & \multicolumn{1}{c}{$|\Out(S)|$} \\ \hline
  $L_2(4)$ & $2^2\cdot 3\cdot 5$ & $2$\\
  $L_2(9)$ & $2^3\cdot 3^2\cdot 5$ & $4$\\
  $L_2(7)$ & $2^3\cdot 3\cdot 7$ & $2$\\
  $L_2(8)$ & $2^3\cdot 3^2\cdot 7$ & $3$\\
  $A_7$ & $2^3\cdot 3^2\cdot 5\cdot 7$ & $2$\\
  $L_3(4)$ & $2^6\cdot 3^2\cdot 5\cdot 7$ & $12$\\
  $L_4(2)$ & $2^6\cdot 3^2\cdot 5\cdot 7$ & $2$\\
  $L_2(13)$ & $2^2\cdot 3\cdot 7\cdot 13$ & $2$\\
  $Sz(8)$ & $2^6\cdot 5\cdot 7\cdot 13$ & $3$\\
  $L_2(64)$ & $2^6\cdot 3^2\cdot 5\cdot 7\cdot 13$ & $6$\\
  $L_2(79)$ & $2^4\cdot 3\cdot 5\cdot 7\cdot 13\cdot 79$ & $2$\\
  $L_3(79)$ & $2^6\cdot 3^2\cdot 5\cdot 7^2\cdot 13^2\cdot 43\cdot 79^3$ & $6$\\
\end{tabular}
\end{table}

If $S$ is isomorphic to a simple group listed in the first column of Table \ref{tbl:79} except $L_3(79)$, then $43\in K$, which is a contradiction. Therefore, $S\cong L_3(79)$ and $L_3(79)\leq G/K\leq \Aut(L_3(79))$. Moreover, since $|G|=|L_3(79)|$, it follows that $|K|=1$, and hence $G\cong L_3(79)$.
\end{proof}

\begin{proposition}\label{prop:11}
If $|G| =|L_{3}(81)|$ and $\D(G)=\D(L_3(81))$, then $G\cong L_3(81)$.
\end{proposition}
\begin{proof}
According to Table ~\ref{tbl:simple}, $|G|=2^9\cdot 3^{12}\cdot 5^2\cdot 7\cdot 13\cdot 41\cdot 73$ and $\D(G)=(3,2,3,2,2,2,2)$. Then the only possible graphs for $\Gamma(G)$ are as in Figure~\ref{fig:81}.
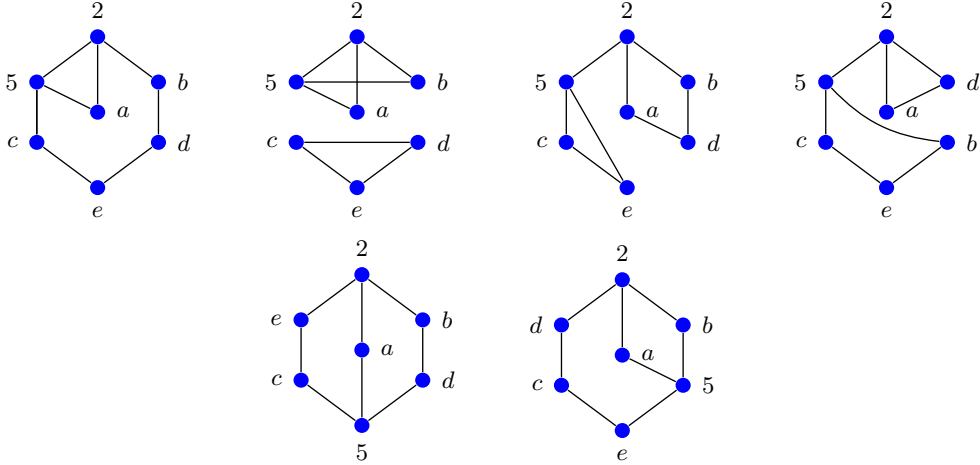
\begin{figure}
  \centering
  \caption{Possibilities for the prime graph of $G$ in Proposition~\ref{prop:11}.}\label{fig:81}
 \begin{subfigure}[t]{0.2\textwidth} 
 \centering
 \begin{tikzpicture}[scale=1]
    \node[vertex] (1) [label=above:{\scriptsize $2$}] at (0,1){};
    \node[vertex] (2) [label=left :{\scriptsize $5$}] at (-.8,0.4){};
    \node[vertex] (3) [label=right:{\scriptsize $a$}] at (0,0) {};
    \node[vertex] (4) [label=right:{\scriptsize $b$}] at (.8,0.4){};
    \node[vertex] (5) [label=left :{\scriptsize $c$}] at (-.8,-.4){};
    \node[vertex] (6) [label=right:{\scriptsize $d$}] at (0.8,-0.4){};
    \node[vertex] (7) [label=below :{\scriptsize $e$}] at (0,-1){};
    \path[edge style]
    (1) edge [] node [] {} (2)
        edge [] node [] {} (3)
        edge [] node [] {} (4)
    (2) edge [] node [] {} (3)
        edge [] node [] {} (5)
    (5) edge [] node [] {} (2)
        edge [] node [] {} (7)
    (6) edge [] node [] {} (7)
        edge [] node [] {} (4);
 \end{tikzpicture}
 \end{subfigure}%
 \quad
 \begin{subfigure}[t]{0.2\textwidth} 
 \centering
 \begin{tikzpicture}[scale=1]
    \node[vertex] (1) [label=above:{\scriptsize $2$}] at (0,1){};
    \node[vertex] (2) [label=left :{\scriptsize $5$}] at (-.8,0.4){};
    \node[vertex] (3) [label=right:{\scriptsize $a$}] at (0,0) {};
    \node[vertex] (4) [label=right:{\scriptsize $b$}] at (.8,0.4){};
    \node[vertex] (5) [label=left :{\scriptsize $c$}] at (-.8,-.4){};
    \node[vertex] (6) [label=right:{\scriptsize $d$}] at (0.8,-0.4){};
    \node[vertex] (7) [label=below :{\scriptsize $e$}] at (0,-1){};
    \path[edge style]
    (1) edge [] node [] {} (2)
        edge [] node [] {} (3)
        edge [] node [] {} (4)
    (2) edge [] node [] {} (3)
        edge [] node [] {} (4)
    (5) edge [] node [] {} (6)
        edge [] node [] {} (7)
    (6) edge [] node [] {} (7);
 \end{tikzpicture}
 \end{subfigure}
  \quad
 \begin{subfigure}[t]{0.2\textwidth} 
 \centering
 \begin{tikzpicture}[scale=1]
    \node[vertex] (1) [label=above:{\scriptsize $2$}] at (0,1){};
    \node[vertex] (2) [label=left :{\scriptsize $5$}] at (-.8,0.4){};
    \node[vertex] (3) [label=right:{\scriptsize $a$}] at (0,0) {};
    \node[vertex] (4) [label=right:{\scriptsize $b$}] at (.8,0.4){};
    \node[vertex] (5) [label=left :{\scriptsize $c$}] at (-.8,-.4){};
    \node[vertex] (6) [label=right:{\scriptsize $d$}] at (0.8,-0.4){};
    \node[vertex] (7) [label=below :{\scriptsize $e$}] at (0,-1){};
    \path[edge style]
    (1) edge [] node [] {} (2)
        edge [] node [] {} (3)
        edge [] node [] {} (4)
    (2) edge [] node [] {} (5)
        edge [] node [] {} (7)
    (5) edge [] node [] {} (7)
    (6) edge [] node [] {} (3)
        edge [] node [] {} (4);
 \end{tikzpicture}
 \end{subfigure}%
  \quad
 \begin{subfigure}[t]{0.2\textwidth} 
 \centering
 \begin{tikzpicture}[scale=1]
    \node[vertex] (1) [label=above:{\scriptsize $2$}] at (0,1){};
    \node[vertex] (2) [label=left :{\scriptsize $5$}] at (-.8,0.4){};
    \node[vertex] (3) [label=right:{\scriptsize $a$}] at (0,0) {};
    \node[vertex] (4) [label=right:{\scriptsize $d$}] at (.8,0.4){};
    \node[vertex] (5) [label=left :{\scriptsize $c$}] at (-.8,-.4){};
    \node[vertex] (6) [label=right:{\scriptsize $b$}] at (0.8,-0.4){};
    \node[vertex] (7) [label=below :{\scriptsize $e$}] at (0,-1){};
    \path[edge style]
    (1) edge [] node [] {} (2)
        edge [] node [] {} (3)
        edge [] node [] {} (4)
    (2) edge [] node [] {} (5)
        edge [bend right=20] node [] {} (6)
    (3) edge [] node [] {} (4)
    (7) edge [] node [] {} (5)
        edge [] node [] {} (6);
 \end{tikzpicture}
 \end{subfigure}
  \quad
 \begin{subfigure}[t]{0.2\textwidth} 
 \centering
 \begin{tikzpicture}[scale=1]
    \node[vertex] (1) [label=above:{\scriptsize $2$}] at (0,1){};
    \node[vertex] (2) [label=left :{\scriptsize $e$}] at (-.8,0.4){};
    \node[vertex] (3) [label=right:{\scriptsize $a$}] at (0,0) {};
    \node[vertex] (4) [label=right:{\scriptsize $b$}] at (.8,0.4){};
    \node[vertex] (5) [label=left :{\scriptsize $c$}] at (-.8,-.4){};
    \node[vertex] (6) [label=right:{\scriptsize $d$}] at (0.8,-0.4){};
    \node[vertex] (7) [label=below :{\scriptsize $5$}] at (0,-1){};
    \path[edge style]
    (1) edge [] node [] {} (2)
        edge [] node [] {} (3)
        edge [] node [] {} (4)
    (4) edge [] node [] {} (6)
    (5) edge [] node [] {} (2)
    (7) edge [] node [] {} (3)
        edge [] node [] {} (5)
        edge [] node [] {} (6);
 \end{tikzpicture}
 \end{subfigure}%
  \quad
 \begin{subfigure}[t]{0.2\textwidth} 
 \centering
 \begin{tikzpicture}[scale=1]
    \node[vertex] (1) [label=above:{\scriptsize $2$}] at (0,1){};
    \node[vertex] (2) [label=left :{\scriptsize $d$}] at (-.8,0.4){};
    \node[vertex] (3) [label=right:{\scriptsize $a$}] at (0,0) {};
    \node[vertex] (4) [label=right:{\scriptsize $b$}] at (.8,0.4){};
    \node[vertex] (5) [label=left :{\scriptsize $c$}] at (-.8,-.4){};
    \node[vertex] (6) [label=right:{\scriptsize $5$}] at (0.8,-0.4){};
    \node[vertex] (7) [label=below :{\scriptsize $e$}] at (0,-1){};
    \path[edge style]
    (1) edge [] node [] {} (2)
        edge [] node [] {} (3)
        edge [] node [] {} (4)
    (5) edge [] node [] {} (2)
        edge [] node [] {} (7)
    (6) edge [] node [] {} (3)
        edge [] node [] {} (4)
        edge [] node [] {} (7);
 \end{tikzpicture}
 \end{subfigure}%
 \end{figure}
In each case, we observe that $\Delta=\{a,b,c\}$ forms an independent set of $\Gamma(G)$, and so $\alpha(G)\geq3$. Note that $\deg(2)=3$ and $|\pi(G)|=7$. Then $\alpha(2,G)\geq2$, and so by Lemma \ref{almost simple}, there is a non-Abelian finite  simple group $S$ such that $S\leq G/K\leq \Aut(S)$, where $K$ is a maximal normal solvable subgroup of $G$. We show that $73\not \in\pi(K)$. Assume the contrary. Then $73\in\pi(K)$. Suppose $p\in\{7,13,41\}$. If $p\in\pi(K)$, then $K$ has a cyclic Hall subgroup of order $p\cdot 73$, and so $p$ is adjacent to $73$, for all $p\in \{7,13,41\}$. If $p\not \in \pi(K)$, then by Frattini argument $G=K\N_G(P)$, where $P$ is a Sylow $73$-subgroup of $K$. Hence $\N_G(P)$ contains an element $x$ of order $p$. Now $P\langle x \rangle$ is a cyclic subgroup of order $p\cdot 73$. This implies that $p$ and $73$ are adjacent. Therefore, $73$ is adjacent to $p$, for all $p\in\{7,13,41\}$, which is a contradiction as the degree of $73$ is $2$. Therefore, $73\notin\pi(K)$. We now prove that $S\cong L_3(81)$. By Lemma \ref{out}, we must have $73\notin\pi(\Out(S))$, and so $73\notin\pi(K)\cup\pi(\Out(S))$. This implies that $73\in\pi(S)$. Now by \cite[Table 1]{ZAVARNITSIN}, $S$ is isomorphic to one of the groups in Table \ref{tbl:81} below. If $S$ is isomorphic to one of the groups $U_3(9)$, $L_2(3^6)$ and $G_2(9)$, then $41\in\pi(K)$, which is a contradiction. Therefore $S\cong L_3(81)$,  and since $L_3(81)\leq G/K\leq \Aut(L_3(81))$ and $|G|=|L_3(81)|$, it follows that $|K|=1$, and hence $G$ is isomorphic to $L_3(81)$.
\end{proof}

\begin{table}[h]
 \centering
 \caption{Non-Abelian finite simple groups $S$ whose order divides $|L_3(81)|$.}\label{tbl:81}
 \small
\begin{tabular}{llc}
  \multicolumn{1}{l}{$S$} & \multicolumn{1}{l}{$|S|$} & \multicolumn{1}{c}{$|\Out(S)|$} \\ \hline
  $U_3(9)$ & $2^5\cdot 3^6\cdot 5^2\cdot 73$ & $2$\\
  $L_2(3^6)$ & $2^3\cdot 3^6\cdot 5\cdot 7\cdot 13\cdot 73$ & $9$\\
  $G_2(9)$ & $2^8\cdot 3^{12}\cdot 5^2\cdot 7\cdot 13\cdot 73$ & $6$\\
  $L_3(81)$ & $2^9\cdot 3^{12}\cdot 5^2\cdot 7\cdot 13\cdot 41\cdot 73$ & $8$\\
\end{tabular}
\end{table}

\begin{proposition}\label{prop:12}
If $|G| =|L_{3}(83)|$ and $\D(G)=\D(L_3(83))$, then $G\cong L_3(83)$.
\end{proposition}
\begin{proof}
According to Table ~\ref{tbl:simple}, $|G|=2^4\cdot 3\cdot 7\cdot 19\cdot 41^2\cdot 83^3\cdot 367$ and $\D(G)=(4,3,3,1,4,2,1)$. By Lemma~\ref{lem:alph3}, we have that  $\alpha(G)\geq3$. Furthermore, $\alpha(2,G)\geq2$ as $\deg(2)=4$ and $|\pi(G)|=7$. By Lemma \ref{almost simple}, there is a non-Abelian finite  simple group $S$ such that $S\leq G/K\leq \Aut(S)$, where $K$ is a maximal normal solvable subgroup of $G$. We show that $\pi(K)$ does not contain $367$. Assume the contrary. Then $367 \in \pi(K)$. Suppose $p\in\{7,19\}$. If $p\in\pi(K)$, then $K$ contains a cyclic Hall subgroup of order $p\cdot 367$, and so $p$ is adjacent to $367$. If $p\not \in \pi(K)$, then by Frattini argument $G=K\N_G(P)$, where $P$ is a Sylow $367$-subgroup of $K$, and so $\N_G(P)$ has an element $x$ of order $p$. Note that $P\langle x \rangle$ is a cyclic subgroup of $G$ of order $p\cdot 367$. Then $p$ is adjacent to $367$. Therefore, both $7$ and $19$ and $367$ are adjacent in $\Gamma(G)$, which is a contradiction. We prove that $S\cong L_3(83)$. By Lemma \ref{out}, $367\notin\pi(\Out(S))$, then $367\notin\pi(K)\cup\pi(\Out(S))$, and so $367\in\pi(S)$. Therefore by \cite[Table 1]{ZAVARNITSIN}, $S$ is isomorphic to $L_3(83)$, and so $L_3(83)\leq G/K\leq \Aut(L_3(83))$. Moreover, since $|G|=|L_3(83)|$, it follows that $|K|=1$, and hence $G$ is isomorphic to $L_3(83)$.
\end{proof}

\noindent \textbf{Proof of Theorem~\ref{thm:main}.} \ The proof of  Theorem \ref{thm:main} follows immediately from \cite{Darafsheh,Rezaeezadeh,Rezaeezadeh1} and Propositions~\ref{prop:1}-\ref{prop:12}.




\begin{thebibliography}{10}

\bibitem{ATLAS}
J.~H. Conway, R.~T. Curtis, S.~P. Norton, R.~A. Parker, and R.~A. Wilson.
\newblock {\em Atlas of finite groups}.
\newblock Oxford University Press, Eynsham, 1985.
\newblock Maximal subgroups and ordinary characters for simple groups, With
  computational assistance from J. G. Thackray.

\bibitem{Gorenstein}
D.~Gorenstein.
\newblock {\em Finite groups}.
\newblock Chelsea Publishing Co., New York, second edition, 1980.

\bibitem{Grechkoseeva}
M.~A. Grechkoseeva, W.~Shi, and A.~V. Vasilev.
\newblock Recognition by spectrum for finite simple groups of {L}ie type.
\newblock {\em Front. Math. China}, 3(2):275--285, 2008.

\bibitem{Higman}
G.~Higman.
\newblock Finite groups in which every element has prime power order.
\newblock {\em J. London Math. Soc.}, 32:335--342, 1957.

\bibitem{Shitian}
S.~Liu.
\newblock O{D}-characterization of some alternating groups.
\newblock {\em Turkish J. Math.}, 39(3):395--407, 2015.

\bibitem{Darafsheh}
A.~R. Moghaddamfar, A.~R. Zokayi, and M.~R. Darafsheh.
\newblock A characterization of finite simple groups by the degrees of vertices
  of their prime graphs.
\newblock {\em Algebra Colloq.}, 12(3):431--442, 2005.

\bibitem{Rezaeezadeh}
G.~R. Rezaeezadeh, M.~Bibak, and M.~Sajjadi.
\newblock Characterization of projective special linear groups in dimension
  three by their orders and degree patterns.
\newblock {\em Bull. Iranian Math. Soc.}, 41(3):551--580, 2015.

\bibitem{Rezaeezadeh1}
G.~R. Rezaeezadeh, M.~R. Darafsheh, M.~Sajjadi, and M.~Bibak.
\newblock O{D}-characterization of almost simple groups related to {$L\sb
  3(25)$}.
\newblock {\em Bull. Iranian Math. Soc.}, 40(3):765--790, 2014.

\bibitem{Vasilev}
A.~V. Vasilev and I.~B. Gorshkov.
\newblock On the recognition of finite simple groups with a connected prime
  graph.
\newblock {\em Sibirsk. Mat. Zh.}, 50(2):292--299, 2009.

\bibitem{Williams}
J.~S. Williams.
\newblock Prime graph components of finite groups.
\newblock {\em J. Algebra}, 69(2):487--513, 1981.

\bibitem{ZAVARNITSIN}
A.~V. {Zavarnitsine}.
\newblock {Finite simple groups with narrow prime spectrum}.
\newblock {\em ArXiv e-prints}, Oct. 2008.

\bibitem{Zhang1}
L.~Zhang and W.~Shi.
\newblock {OD}-characterization of all simple groups whose orders are less than
  108.
\newblock {\em Frontiers of Mathematics in China}, 3(3):461--474, 2008.

\end{thebibliography}

\end{document}